\documentclass{article}

\usepackage{graphicx}
\usepackage{amsmath}
\usepackage{amsthm}
\usepackage{amssymb}
\usepackage{amsfonts}
\usepackage{latexsym}
\usepackage[labelsep=period]{caption}
\usepackage{caption}
\usepackage{hyperref}
\usepackage{url}
\usepackage{subcaption}
\usepackage{comment}
\usepackage{wasysym}
\usepackage{fontawesome5}
\usepackage{tikzsymbols}
\usepackage{mathtools}
\usepackage{color}
\usepackage{mathrsfs}

\def\dj{d\kern-0.4em\char"16\kern-0.1em}
\def\Dj{\mbox{\raise0.3ex\hbox{-}\kern-0.4em D}}

\newtheorem{theorem}{Theorem}[section]
\newtheorem{lemma}[theorem]{Lemma}
\newtheorem{remark}[theorem]{Remark}

\theoremstyle{definition}
\newtheorem{definition}[theorem]{Definition}

\let\tmp\oddsidemargin
\let\oddsidemargin\evensidemargin
\let\evensidemargin\tmp
\reversemarginpar

\usepackage{graphicx} 

\title{An interpolation problem for extended Gevrey regularity}
\author{
Jelena Dimitri\'c\thanks{University of Novi Sad, Serbia, email: jelena222000@gmail.com}
\and
\Dj or\dj e Vu\v ckovi\'c\thanks{Technical Faculty "Mihajlo Pupin", University of Novi Sad, Serbia, email: djordjeplusja@gmail.com}
\and
Milica \v Zigi\'c\thanks{Faculty of Sciences, University of Novi Sad, Serbia, email: milica.zigic@dmi.uns.ac.rs}
}

\date{\today}

\begin{document}

\maketitle

\begin{abstract}
We study an interpolation problem in extended Gevrey classes defined by the sequences
$M_n^{\tau,\sigma}=n^{\tau n^\sigma}$, $n\in\mathbb N$, $\tau>0$, $\sigma>1$.
We show that, under a suitable growth condition on a divergent sequence of derivative orders, estimates of extended Gevrey type imposed only along this sequence imply corresponding estimates for all derivative orders. In particular, we explicitly quantify the change of the parameters in the resulting estimates. To this end, we first establish an equivalence between two definitions of the extended Gevrey classes, one involving the supergeometric factor $h^{n^\sigma}$ and another in which this factor is omitted. We then establish a corresponding interpolation principle for extended Gelfand--Shilov spaces, including a symmetric characterization and a formulation in terms of the associated function of the defining sequence.
\end{abstract}



\section{Introduction}

The present paper is devoted to an interpolation problem for extended Gevrey classes. The underlying question is whether ultradifferentiable regularity can be inferred from information on the derivatives only at a suitably chosen increasing sequence of orders. Such lacunary information naturally arises in regularity questions for solutions of various classes of partial differential equations and related problems, where one seeks to determine which derivatives need to be controlled in order to conclude that a function belongs to a prescribed ultradifferentiable class. This type of question also appears in various contexts involving regularity and analyticity; see \cite{PAMM2022,RS2024} and the references therein.

The Gevrey classes occupy a central position in the theory of partial differential equations, providing a natural scale of regularity between real-analytic and smooth functions. They arise naturally in the study of hypoellipticity, local solvability, propagation of singularities, and well-posedness for various classes of partial differential equations; see \cite{Gevrey,Rodino}. In many situations, however, Gevrey regularity proves to be too restrictive, whereas $C^\infty$-regularity is too weak, motivating the study of intermediate classes of ultradifferentiable functions. For instance, J\'ez\'equel \cite{J} proved that the trace formula for Anosov flows in dynamical systems holds for certain intermediate regularity classes, while Cicognani and Lorenz considered another intermediate regularity scale in their study of the well-posedness of strictly hyperbolic equations \cite{CL}. More recently, C\'aceres and Lastra \cite{CL2026} studied moment differential equations associated with moments of intermediate growth and proved existence and uniqueness of their solutions.

A particularly relevant example of such intermediate regularity classes is given by the extended Gevrey classes. A systematic approach to these classes was introduced in \cite{PTT1,PTT} and further developed and systematized in \cite{TTZ2024} through the two-parameter family of sequences
$$
M_n^{\tau,\sigma}
=n^{\tau n^\sigma},
\qquad n\in\mathbb N,\qquad
M_0^{\tau,\sigma}=1,\qquad
\tau>0,\quad \sigma>1,
$$
which define the extended Gevrey classes. These classes contain all Gevrey classes while remaining strictly smaller than $C^\infty$, thus providing a finer scale of intermediate regularity. They have recently attracted considerable attention both in the general theory of ultradifferentiable functions and in applications to partial differential equations and microlocal analysis; see \cite{CL2026,CL,J,Javier02,Javier01,TTT-24}. In particular, it has recently been shown that the extended Gevrey classes naturally fit into the framework of weight matrix classes \cite{Javier01,TT-01}, despite the fact that their defining sequences do not satisfy Komatsu's moderate growth condition
$$
\displaystyle (M.2)\quad 
(\exists\, C>0)\qquad
M_{n+m}\leq C^{n+m+1}M_nM_m,
\qquad n,m\in\mathbb N.
$$

We now turn to the interpolation problem. An interpolation result asserting that control of derivatives along a prescribed sequence of orders implies control at all orders was established in \cite{PAMM2022} for non-quasianalytic Denjoy--Carleman classes of Roumieu type, under suitable assumptions on the defining sequence and on the sequence of derivative orders. The condition imposed in \cite{PAMM2022} is not satisfied by the extended Gevrey sequences considered here. A weaker condition was recently used in \cite{RS2024}, but it is still not satisfied in our setting. On the other hand, the general interpolation framework developed in Chapter~5 of \cite{RS2024} includes the extended Gevrey sequences and therefore yields the existence of the interpolation result considered here at a qualitative level, see Remark \ref{rem 3.6}.

Our main contribution is to quantify this interpolation phenomenon in the extended Gevrey setting. Namely, let $(d_n)_{n\in\mathbb N_0}$ be a strictly increasing, divergent sequence of positive integers satisfying
$$
d_{n+1}\leq Ld_n,\qquad n\in\mathbb N_0,
$$
for some constant $L>1$. We prove that if the derivatives of a smooth function are controlled only at the orders $d_n$, $n\in\mathbb N_0$, by extended Gevrey-type estimates, then the derivatives of all orders satisfy the corresponding extended Gevrey bounds. More precisely, we explicitly determine the parameters of the estimates for derivatives of all orders in terms of those controlling the derivatives at the orders $d_n$. Thus, we quantify the change of the interpolation parameters rather than merely asserting their existence.

We first establish an equivalence between two definitions of the extended Gevrey classes: one involving the supergeometric factor $h^{n^\sigma}$ and another in which this factor is omitted. Theorem~\ref{thm:with-or-without-h} provides this equivalence and is essential for the subsequent arguments. It also allows us to relate the extended Gevrey setting to the classical Carleman and Braun--Meise--Taylor frameworks, whose defining estimates involve the geometric scaling factor $h^n$. We then establish the quantitative interpolation theorem for the extended Gevrey classes and derive the corresponding results for the extended Gelfand--Shilov spaces. In the latter setting, we obtain a symmetric characterization in terms of the two families of estimates defining these spaces, as well as a characterization in terms of the associated function of the defining sequence.

We emphasize that the extended Gevrey classes are defined via the local behavior of their elements, i.e., for functions defined on an interval, whereas the extended Gelfand--Shilov spaces are introduced by imposing suitable global estimates. However, the essential estimates used in the proof of the interpolation principle for the extended Gevrey case depend solely on the length of the considered interval and are completely independent of its location. This feature allows the local result to be transferred to the global setting of the extended Gelfand--Shilov spaces.

The paper is organized as follows. In Section~\ref{sec:2}, we collect the necessary background and auxiliary estimates concerning the extended Gevrey classes. We also state and prove Theorem~\ref{thm:with-or-without-h}, which is used throughout the paper. In Section~\ref{sec:3}, we prove the main result of the paper, Theorem~\ref{thm: main}, concerning the interpolation problem in the extended Gevrey spaces. Section~\ref{sec:4} is devoted to the corresponding results for the extended Gelfand--Shilov spaces. We first establish Theorem~\ref{12ekv} and then derive the corresponding statement for the associated function, given in Lemma~\ref{15ekv}. Finally, an appendix contains the proofs of two technical lemmas.

\section{Extended Gevrey spaces}
\label{sec:2}

In this section, we recall the definition of the extended Gevrey classes and discuss their topological properties, which will be used in the sequel. We also show that the defining sequences $(M_n^{\tau,\sigma})_{n\in\mathbb N_0}$ given by \eqref{eq:the-sequence} do not satisfy the condition \eqref{condition X} considered in \cite{PAMM2022}, nor the weaker condition \eqref{condition bounded} assumed in \cite{RS2024}.

\subsection{Defining weight sequences}

Let $\tau>0$ and $\sigma>1$. We consider the weight sequence
$(M_n^{\tau,\sigma})_{n\in\mathbb N_0}$ defined by
\begin{equation} \label{eq:the-sequence}
    M_{n}^{\tau,\sigma}=n^{\tau n^\sigma}, \quad n\in \mathbb{N},
    \qquad M_0^{\tau,\sigma}=1.
\end{equation}

The main properties of $(M_{n}^{\tau,\sigma})_{n\in\mathbb N_0}$ are summarized in Lemma~\ref{osobineM_p_s} (cf. \cite[Lemma~2]{TTZ2024}).

\begin{lemma} \label{osobineM_p_s}
Let $\tau>0$, $\sigma>1$, $M_0^{\tau,\sigma}=1$, and $M_n^{\tau,\sigma}=n^{\tau n^{\sigma}}$ for $n\in \mathbb N$. Then the following properties hold:
\begin{itemize}
\item[$(M.1)$] \hspace{1em} 
$ (M_n^{\tau,\sigma})^2\leq M_{n-1}^{\tau,\sigma}M_{n+1}^{\tau,\sigma}, \quad n\in \mathbb N,$ \medskip

\item[$\widetilde{(M.2)}$]  \hspace{1em}   
$ M_{n+m}^{\tau,\sigma}\leq C^{n^{\sigma} + m^{\sigma}}
M_n^{\tau 2^{\sigma-1},\sigma}M_m^{\tau 2^{\sigma-1},\sigma}, \quad n,m\in \mathbb N_0,$
for some constant $ C\geq 1$, \medskip

\item[$\widetilde{(M.2)'}$] \hspace{1em}
$ M_{n+1}^{\tau,\sigma}\leq C ^{n^{\sigma}} M_n^{\tau,\sigma}, \quad n\in \mathbb N_0,$
for some constant $C\geq 1$, \medskip

\item[$(M.3)'$] \hspace{1em}
$ \displaystyle \sum_{n=1}^{\infty}\frac{M_{n-1}^{\tau,\sigma}}{M_n^{\tau,\sigma}} <\infty.$
\end{itemize}
\end{lemma}

Compared with the classical Komatsu setting, the sequences $(M_n^{\tau,\sigma})_{n\in\mathbb N_0}$ 
may
exhibit a stronger, superexponential growth behaviour governed by the parameter $\sigma>1$. In contrast to the classical Komatsu setting, they do not satisfy Komatsu’s condition
\[
(M.2)\qquad (\exists C>0)\; M_{n+m} \leq C^{n+m+1} M_n M_m,
\quad n,m\in\mathbb N.
\]
Thus, the sequences $(M_n^{\tau,\sigma})_{n\in\mathbb N_0}$ fall outside the standard Komatsu setting.

\begin{remark}
    Note that, for $\sigma=1$, one obtains $\displaystyle M^{\tau,1}_p = p^{\tau p}$, $p\in\mathbb N$, where $\tau>0$, which corresponds to the Gevrey sequence of order $\tau>0$. Then the conditions $\widetilde{(M.2)'}$ and $\widetilde{(M.2)}$ reduce to the classical Komatsu conditions  ${(M.2)'}$ and ${(M.2)}$, respectively. Recall,
$$(M.2)'\quad (\exists\, C>0)\ M_{p+1} \leq C^{p+1} M_{p},\quad p\in \mathbb N.$$
Moreover, the Gevrey sequences satisfy \emph{non-quasianalyticity} condition ${(M.3)'}$ if and only if $\tau > 1$. 
\end{remark}

We now compare the defining sequence $(M_n^{\tau,\sigma})_{n\in\mathbb N_0}$ with those considered in \cite{PAMM2022}. There, the authors assume that the defining sequence satisfies three conditions that are crucial for their main results. Of these, the sequence $(M_n^{\tau,\sigma})_{n\in\mathbb N_0}$ satisfies condition $(M.1)$ and the lower bound
\[
c^n n^n \leq M_n,\quad c>0,
\]
ensuring the inclusion of real-analytic functions in the corresponding function space.

The remaining condition is
\begin{equation}\label{condition X}
(\exists m_0\geq 0)\quad
M_j \leq M_k^{j/k} M_i^{j/i},
\quad \text{for } i,k>m_0 \text{ with } i<j \text{ and } \frac{j}{i}<k.
\end{equation}

We show that the sequence $(M_n^{\tau,\sigma})_{n\in\mathbb N_0}$ does not satisfy this condition. Let $\sigma>1$ and, without loss of generality, set $\tau=1$. Then $M_n^{1,\sigma}=n^{n^{\sigma}}$, $n\in\mathbb N$. Assume, by contradiction, that $(M_n^{1,\sigma})_{n\in\mathbb N_0}$ satisfies \eqref{condition X}. Taking the natural logarithm of both sides of \eqref{condition X} and dividing by $j$, we obtain the equivalent inequality
\begin{equation}\label{condition log}
(\exists m_0\geq 0)\quad
j^{\sigma-1}\ln j \leq k^{\sigma-1}\ln k+i^{\sigma-1}\ln i,\  i,k>m_0,\  i<j,\  \frac{j}{i}<k.
\end{equation}
Since  
\begin{equation*}
\lim_{m\to\infty}\frac{m^{2\sigma-2} \ln m^2}{2(m+1)^{\sigma-1} \ln (m+1)}=\infty,    
\end{equation*}
there exists $m_0$ (depending on $\sigma$) such that for every $m\geq m_0$ we have
\begin{equation}
    \label{lim}
    \frac{m^{2\sigma-2} \ln m^2}{2(m+1)^{\sigma-1} \ln (m+1)}>1.
\end{equation}

Now set $i=m$, $k=m+1$, and $j=m^2$ for every 
$m\geq \max\{2,m_0\}$. 
Since the function 
$x\mapsto x^{\sigma-1}\ln x$ is increasing on $[1,\infty)$ and using \eqref{lim}, the right-hand side of \eqref{condition log} satisfies,
\[
\begin{aligned}
k^{\sigma-1}\ln k+i^{\sigma-1}\ln i
&=(m+1)^{\sigma-1}\ln(m+1)+m^{\sigma-1}\ln m\\
&<2(m+1)^{\sigma-1}\ln(m+1)\\
&<m^{2\sigma-2}\ln m^2
=j^{\sigma-1}\ln j.
\end{aligned}
\]
Hence, for this choice of $i$, $j$, and $k$, the right-hand side of \eqref{condition log} is strictly smaller than its left-hand side, but this contradicts \eqref{condition log}.
Therefore, \eqref{condition X} does not hold for $(M_n^{\tau,\sigma})_{n\in\mathbb N_0}$.

We further note that, although the sufficient condition \eqref{condition log} is relaxed in \cite[Theorem 4.1]{RS2024}, our defining sequence $(M_n^{\tau,\sigma})_{n\in\mathbb N_0}$ still does not satisfy this weaker condition. More precisely, the condition imposed in \cite[Theorem 4.1]{RS2024} requires that, for a strictly increasing infinite sequence $(d_n)_{n\in\mathbb N_0}$ of positive integers,
\begin{equation}\label{condition bounded}
\sup_{n\in\mathbb N_0}
\frac{m_{d_{n+1}}}{m_{d_n}}<\infty,
\qquad
m_n=(M_n^{\tau,\sigma})^{1/n}.
\end{equation}
We show that this condition fails when $(d_n)_{n\in\mathbb N_0}$ is a geometric sequence. More precisely, let $\tau>0$, $\sigma>1$, and let $d_{n+1}=Ld_n$ for every $n\in\mathbb N_0$, where $L\in\mathbb N\setminus\{1\}$. Then the quotient $m_{d_{n+1}}/m_{d_n}$ is unbounded. Indeed, since
$$
m_n=n^{\tau n^{\sigma-1}},\qquad n\in\mathbb N,
$$
we obtain
$$
\frac{m_{Ld_n}}{m_{d_n}}
=\frac{(Ld_n)^{\tau(Ld_n)^{\sigma-1}}}
{d_n^{\tau d_n^{\sigma-1}}}
=L^{\tau(Ld_n)^{\sigma-1}}
d_n^{\tau(L^{\sigma-1}-1)d_n^{\sigma-1}}.
$$
Since $L>1$, $\tau>0$, and $\sigma>1$, both factors on the right-hand side tend to infinity as $n\to\infty$. Consequently, the boundedness condition \eqref{condition bounded} is not satisfied.

\subsection{The spaces $\mathcal E_{\sigma} ([a,b])$}

In this section, we define the extended Gevrey spaces in two equivalent ways, which differ only in whether a scaling supergeometric factor is incorporated into their defining estimates.

\begin{definition}\label{Def:extGS}
Let $\sigma>1$. 
The extended Gevrey space of Roumieu type related to $(M_{n}^{\tau ,\sigma})_{n\in \mathbb{N}_0}$ is given by
\begin{multline*}
\mathcal E_{\{\sigma\}} ([a,b])=\Big\{f\in  C^{\infty}([a,b]) \colon  (\exists \tau>0)\;(\exists C>0)\\
   \max_{x\in[a,b]}|f^{(n)}(x)|\leq C M^{\tau,\sigma}_{n}, \;n \in\mathbb N_0\Big\}
\end{multline*}
and the extended Gevrey space of Beurling type related to $(M_{n}^{\tau ,\sigma})_{n\in \mathbb{N}_0}$ is given by
\begin{multline*}
 \mathcal E_{(\sigma)} ([a,b])=\Big\{f\in C^{\infty}([a,b]) \colon
 (\forall \tau>0) (\exists C>0)\\ \max_{x\in[a,b]}|f^{(n)}(x)|\leq C M^{\tau,\sigma}_{n}, \;n \in\mathbb N_0\Big\},
\end{multline*}
where $(M_{n}^{\tau ,\sigma})_{n\in \mathbb{N}_0}$ is given by  \eqref{eq:the-sequence}.
\end{definition}

We use $\mathcal E_{\sigma}([a,b])$ as a common notation for
$ \mathcal E_{\{\sigma\}}([a,b])$
or $ \mathcal E_{(\sigma)}([a,b])$.

A detailed overview of the properties of extended Gevrey classes and the corresponding defining sequences can be found in \cite{TTZ2024}.

We now state the following lemma (see \cite[Lemma~2]{TTZ2025}), which will be frequently used in the sequel.

\begin{lemma} \label{lm1}
\begin{equation*} \label{eq:lema-from-ptt}
 \sup_{\rho >0}  \frac{h^{\rho^\sigma}}{\rho^{\tau \rho ^\sigma}} 
 = e^{\frac{\tau}{\sigma e} h^{\frac{\sigma}{\tau}}}
  \end{equation*}
for any $h>0$, $\tau > 0$ and $\sigma >1$.
  \end{lemma}
  
In particular, Lemma \ref{lm1} implies the following useful inequality: for any given $h > 0$,
\begin{equation}\label{jede h}
    h ^{n^\sigma}  M^{\frac{\tau}{2}, \sigma} _{n} \leq
 A   M^{\tau, \sigma} _{n}, \qquad  n \in \mathbb{N}_0 ,
 \end{equation}
where the constant $A>0$ depends on $\sigma$, $\tau$  and $h$.

In the literature, it is common to define Gevrey spaces using the geometric factor $h^n$. In what follows, we introduce the corresponding definition of extended Gevrey spaces in this setting and show that the two definitions are equivalent. To this end, we introduce the notation:
\begin{multline} \label{eq:GS-with-factor}
   \mathscr E_{\{\sigma \} }([a,b])  =
    \Big\{ f \in C^{\infty} ([a,b]) \colon
    (\exists \tau>0) \;(\exists h>0)\;(\exists C >0) \\
    \max_{x\in[a,b]}|f^{(n)}(x)| \leq 
   C h^{n^\sigma } M^{\tau, \sigma}_{n},
   \; n\in\mathbb N_0
   \Big\}
\end{multline}
and 
\begin{multline*} 
    \mathscr E_{ (\sigma )} ([a,b])  =
    \Big\{ f \in C^{\infty} ([a,b]) \colon
    (\forall \tau>0) \;(\forall h>0)\;(\exists C >0) \\
    \max_{x\in[a,b]}|f^{(n)}(x)| \leq 
   C h^{n^\sigma } M^{\tau, \sigma}_{n},
   \;n\in\mathbb N_0
   \Big\}.
\end{multline*}

It is clear that $\mathcal E_{\{\sigma\}}([a,b])\subseteq \mathscr E_{\{\sigma \}} ([a,b])$ and that $\mathcal E_{(\sigma)}([a,b])\supseteq \mathscr E_{ (\sigma )} ([a,b])$. We show that the opposite inclusions also hold and, moreover, that these spaces coincide as locally convex spaces.

\begin{theorem} \label{thm:with-or-without-h}
Let $\sigma>1$ and let $M_0 ^{\tau,\sigma}=1$,
and $M_n ^{\tau,\sigma}=n^{\tau n^{\sigma}}$, $n\in \mathbb N$.
Then 
\begin{equation*}
\mathcal E_{\sigma }([a,b])
= \mathscr E_{\sigma }([a,b]),    
\end{equation*}
 as locally convex spaces.
\end{theorem}

The proof of Theorem \ref{thm:with-or-without-h} follows as a consequence of Lemma \ref{lm1} by the same argument as in \cite[Theorem~1]{TTZ2025}. However, in Section \ref{subsec 2.3}, we present an alternative proof based on topological arguments.

\begin{remark}
   It follows from the proof of \cite[Theorem~1]{TTZ2025} that in the definition of $\mathscr E_{(\sigma)}([a,b])$ given by \eqref{eq:GS-with-factor}, the condition $(\forall h>0)$ can be equivalently replaced by $(\exists h>0)$.
\end{remark}

\begin{remark}
 Here, we emphasize another important difference from the setting in \cite{PAMM2022} and \cite{RS2024}, where the spaces under consideration are defined using the geometric factor $h^n$. In our case, however, the factor $h^{n^\sigma}$ appears instead, leading to essentially different techniques and ideas in the proofs.
\end{remark}

\subsection{Topological properties of extended Gevrey classes}
\label{subsec 2.3}

In this section, we present the proof of Theorem \ref{thm:with-or-without-h}. This result is important as it will be used throughout the paper. However, its proof is of a topological nature and would distract from the main line of exposition. To this end, we first discuss the topological structure of the extended Gevrey spaces. Moreover, this section is included for the sake of completeness, since, to the best of our knowledge, Theorem \ref{thm:with-or-without-h} has not been explicitly formulated and proved elsewhere in the literature.

We start by defining constructive spaces.

\begin{definition}
Let $\sigma>1$, $\tau>0$, $h>0$, and let the defining weight sequence $(M_{n}^{\tau ,\sigma})_{n\in \mathbb{N}_0}$ be given by \eqref{eq:the-sequence}. Then we define the space (cf. \cite{PTT})
\[
\mathcal E_{h,\tau,\sigma}([a,b])
:=
\left\{
f\in C^{\infty}([a,b]) :
\|f\|_{h,\tau,\sigma}=\sup_{n\in\mathbb{N}_0} \frac{ \max_{x\in[a,b]}|f^{(n)}(x)|}{ h^{n^{\sigma}}M^{\tau,\sigma}_{n}} < \infty
\right\}.
\]
\end{definition}

These are Banach spaces with respect to the norm $\|\cdot\|_{h,\tau,\sigma}$. We can then write the extended Gevrey spaces from Definition \ref{Def:extGS} as follows.

Let $\sigma>1$ and let the defining weight sequence $(M_{n}^{\tau ,\sigma})_{n\in \mathbb{N}_0}$ be given by \eqref{eq:the-sequence}. Then the extended Gevrey space of Roumieu type is given by
\begin{equation*}
\mathcal E_{\{\sigma\}} ([a,b])
=\bigcup_{\tau>0}\mathcal E_{1,\tau,\sigma}([a,b]),
\end{equation*}
endowed with the inductive limit topology, and the extended Gevrey space of Beurling type is given by
\begin{equation*}
\mathcal E_{(\sigma)} ([a,b])
=\bigcap_{\tau>0} \mathcal E_{1,\tau,\sigma}([a,b]),
\end{equation*}
endowed with the projective limit topology.

The following result will be used in the sequel. We omit the proof, since it follows directly from \cite[Proposition~2.1]{PTT}.

\begin{lemma}\label{tau12}
Let $\sigma>1$, $0<\tau_1<\tau_2$, and $h_1,h_2>0$ be arbitrary. Then the canonical inclusion
\[
\mathcal E_{h_1,\tau_1,\sigma}([a,b]) \hookrightarrow \mathcal E_{h_2,\tau_2,\sigma}([a,b])
\]
is continuous.
\end{lemma}

Let $\sigma>1$ and $0<\tau_1<\tau_2$. The canonical inclusion
$\mathcal E_{1,\tau_1,\sigma}([a,b])\hookrightarrow \mathcal E_{2,\tau_1,\sigma}([a,b])$
is continuous and nuclear (see \cite[Theorem~3.1]{PTT1}). From the previous lemma it follows that the embedding
\[
\mathcal E_{2,\tau_1,\sigma}([a,b]) \hookrightarrow \mathcal E_{1,\tau_2,\sigma}([a,b])
\]
is continuous, and therefore the composition
\[
\mathcal E_{1,\tau_1,\sigma}([a,b])
\hookrightarrow
\mathcal E_{2,\tau_1,\sigma}([a,b])
\hookrightarrow
\mathcal E_{1,\tau_2,\sigma}([a,b])
\]
is nuclear, as a composition of a nuclear and a continuous mapping (see \cite[Proposition~47.1]{Tr}).

In view of the obvious inclusions $\mathcal{E}_{1,\tau_1,\sigma}([a,b])\subset \mathcal{E}_{1,\tau_2,\sigma}([a,b])$, $\tau_1\leq\tau_2$, the extended Gevrey spaces can be represented as the union and, respectively, the intersection of countable families of Banach spaces,  namely 
$$\mathcal E_{\{\sigma\}} ([a,b])
=\bigcup_{n\in\mathbb N}\mathcal{E}_{1,n,\sigma}([a,b])\;\;\text{ and } \;\;\mathcal E_{(\sigma)} ([a,b]) =\bigcap_{n\in\mathbb N}\mathcal{E}_{1,\frac{1}{n},\sigma}([a,b]).$$
 
Hence, by \cite[Theorem~25.1]{MV}, the space $\mathcal E_{(\sigma)} ([a,b])$ is a Fréchet space (or simply an $(F)$-space). Moreover, \cite[Proposition~28.4]{MV} implies that $\mathcal E_{(\sigma)}([a,b])$ is nuclear, that is, an $(FN)$-space (Fréchet nuclear space).

Furthermore, by \cite[Theorem~25.19]{MV}, the space $\mathcal E_{\{\sigma\}}([a,b])$ is the strong dual of a Fréchet space, i.e., a $(DF)$-space. In addition, \cite[Section~4.1.1]{Pietsch} shows that $\mathcal E_{\{\sigma\}}([a,b])$ is a $(DFN)$-space, that is, the strong dual of a Fréchet nuclear space.

Note that in Definition \ref{Def:extGS} the value $h=1$ was fixed in the definition of the spaces $\mathcal E_{\sigma}([a,b])$ of both Roumieu and Beurling type. In the sequel, we show that taking unions and intersections with respect to $h>0$, endowed with the corresponding inductive and projective limit topologies, does not lead to new spaces.

To this end, we first observe that
\begin{equation*}\label{eq:GS-with-factor sa h}
   \mathscr E_{\{\sigma\}}([a,b])
   = \bigcup_{\tau>0}\bigcup_{h>0} \mathcal E_{h,\tau,\sigma}([a,b]),
\end{equation*}
and
\begin{equation*}
   \mathscr E_{(\sigma)}([a,b])
   = \bigcap_{\tau>0}\bigcap_{h>0} \mathcal E_{h,\tau,\sigma}([a,b]),
\end{equation*}
equipped with the inductive (respectively projective) limit topology. As in the case of $\mathcal E_{\sigma}([a,b])$, these spaces are (DFN)-spaces (respectively (FN)-spaces).

Now we can prove Theorem \ref{thm:with-or-without-h}.

\begin{proof}[\textbf{Proof of Theorem \ref{thm:with-or-without-h}}]
We consider only the Roumieu case. The Beurling case follows by analogous arguments using properties of the projective limit topology.

Let $h>0$ be arbitrary. From Lemma \ref{tau12} it follows that the canonical inclusion
$$\mathcal E_{h,\frac{\tau}{2},\sigma}([a,b]) \hookrightarrow \mathcal E_{1,\tau,\sigma}([a,b])$$
is continuous. Taking the inductive limit with respect to $\tau$ yields the continuity of the embedding
$$\mathcal E_{h,\frac{\tau}{2},\sigma}([a,b]) \hookrightarrow \mathcal E_{\{\sigma\}}([a,b]).$$
Since this holds for every $h,\tau>0$, taking the union over all such parameters yields
\[
\mathscr E_{\{\sigma\}}([a,b]) \subseteq \mathcal E_{\{\sigma\}}([a,b]).
\]
Moreover, the universal property of the inductive limit topology implies that this embedding is continuous. The opposite inclusion follows by the same argument.
\end{proof}

\section{Interpolation problem in extended Gevrey spaces}
\label{sec:3}

In this section, we establish the main result of the paper, motivated by the interpolation results in \cite{PAMM2022,RS2024}. We focus on quantifying the parameter changes arising in the interpolation process, explicitly determining the parameters that yield control of derivatives of all orders from those controlling the derivatives at prescribed, geometrically spaced orders.

The proof of the main theorem relies on the following classical inequality of Cartan and Gorny, which provides an estimate for intermediate derivatives in terms of derivatives with known behavior (see \cite{Cartan, Gorny}).

\begin{lemma}\label{cglemma}
	Let $m\in \mathbb{N}$ such that $m\geq 2$. Let $g\in C^{m}([a,b])$ and set $$G_{j}=\max_{x\in [a,b]}|g^{(j)}(x)|,\quad j=0,\dots,m.$$ Then, for every $l\in\{1,\dots,m-1\}$ 
    \begin{equation}\label{cart gorny eq}
			G_{l}\leq 2\bigg(\dfrac{e^{2}m}{l}\bigg)^{l}G_{0}^{1-\frac{l}{m}}\bigg(\max\bigg\{m!G_{0}\bigg(\dfrac{2}{b-a}\bigg)^{m},G_{m}\bigg\}\bigg)^{\frac{l}{m}}.
	\end{equation}
\end{lemma}

We shall also need the following two auxiliary results. Since their proofs are purely technical, they are postponed to the Appendix.

\begin{lemma}\label{max1}
	Let $D>e$. The function $\phi:[1,\infty)\to \mathbb{R}^{+}$, defined by $\phi(x)=\left(\dfrac{D}{x}\right)^{x}$, attains a unique maximum at $x=\dfrac{D}{e}$.
\end{lemma}

\begin{lemma}\label{max2}
		Let $\sigma>1$. Then 
        \begin{equation*}
			x\ln x\leq Rx^{\sigma},  \quad x\geq 1,\quad\text{where}\quad R=\frac{1}{e(\sigma-1)}.
		\end{equation*}
\end{lemma}
	
\medskip

    We can now formulate and prove the main theorem of this paper.
	
	\begin{theorem}\label{thm: main} 
		Let $\sigma>1$ be given. Let $f\in C^{\infty}([a,b])$, and let $(d_n)_{n\in\mathbb{N}_0}$ be an increasing divergent sequence of positive integers satisfying
        \begin{equation}\label{gap}
			(\exists L>1) \quad \dfrac{d_{n+1}}{d_{n}}\leq L, \quad n\in \mathbb{N}_0.
		\end{equation}
		Assume that \begin{equation}\label{condition}
			(\exists \tau_0 >0)(\exists C_0>0) \quad \max_{x\in[a,b]}|f^{(d_n)}(x)|\leq C_0 M_{d_{n}}^{\tau_0 , \sigma}, \quad n\in \mathbb{N}_0.
		\end{equation}
        Then $f\in \mathcal E_{\{\sigma\}}([a,b])$. More precisely,
         \begin{equation*}
			\max_{x\in[a,b]}|f^{(n)}(x)|\leq C M_{n}^{\tau , \sigma}, \quad n\in \mathbb{N}_0,
		\end{equation*}
        where $\tau=2\tau_0L^\sigma$ and $C=2C_0 A$, with $A$ as in \eqref{jede h}.
	\end{theorem}
	
	\begin{proof} 
    For a smooth function $f:[a,b]\subset \mathbb{R} \rightarrow \mathbb{R},$ we define 
\begin{equation*}
	F_{n}=\max_{x\in [a,b]}|f^{(n)}(x)|,\quad  n\in \mathbb{N}_0.
\end{equation*} 
    To prove the theorem, it is sufficient to show that there exist constants $C>0$ and $\tau>0$ such that $F_k\leq CM_{k}^{\tau,\sigma}$ for all $k\in\mathbb{N}_0$.
		Fix $k\in\mathbb N_0$. Choose $n\in\mathbb N_0$ so that $d_{n}\leq k < d_{n+1}.$ If $k=d_{n}$, the claim follows immediately from (\ref{condition}). Hence, we restrict our attention to the case $d_{n}<k<d_{n+1}$.
		Let \begin{equation*}
			\theta =\dfrac{d_{n+1}-k}{d_{n+1}-d_{n}}\qquad\text{and hence}\qquad 1-\theta=
            \dfrac{k-d_{n}}{d_{n+1}-d_{n}}.
		\end{equation*} 
	Clearly, 
    $0<\dfrac{d_{n+1}-k}{d_{n+1}-d_{n}}<1,$ that is $\theta \in (0,1)$.
	
    In what follows, we make use of the Cartan-Gorny inequality \eqref{cart gorny eq}. For this purpose, choose 
    $$g=f^{(d_{n})},\;\; m=d_{n+1}-d_{n} \;\;\text{ and } \;\;l=k-d_{n}$$ Hence, $G_{0}=F_{d_{n}}$, $G_{l}=F_{k}$ and $G_{m}=F_{d_{n+1}}$. 
	\medskip
	Now, in view of \eqref{cart gorny eq} it follows that 
    \begin{equation*}
		F_{k}\leq 2\bigg(\dfrac{e^{2}(d_{n+1}-d_{n})}{k-d_{n}}\bigg)^{k-d_{n}}F_{d_{n}}^{\theta}\max \bigg\{(d_{n+1}-d_{n})!F_{d_{n}}\bigg( \frac{2}{b-a}\bigg)^{d_{n+1}-d_{n}},F_{d_{n+1}}\bigg\}^{1-\theta}.
	\end{equation*}

There are two cases to consider.

{\bf Case 1.}
\begin{equation*}
	\max \bigg\{(d_{n+1}-d_{n})!F_{d_{n}}\bigg( \frac{2}{b-a}\bigg)^{d_{n+1}-d_{n}},F_{d_{n+1}}\bigg\}=(d_{n+1}-d_{n})!F_{d_{n}}\bigg( \frac{2}{b-a}\bigg)^{d_{n+1}-d_{n}}.
\end{equation*} Then we have
 \begin{eqnarray*}
	F_{k}&\leq& 
    2\bigg(\dfrac{e^{2}(d_{n+1}-d_{n})}{k-d_{n}}\bigg)^{k-d_{n}}\bigg((d_{n+1}-d_{n})!\bigg( \frac{2}{b-a}\bigg)^{d_{n+1}-d_{n}}\bigg)^{1-\theta}F_{d_{n}}.
		\end{eqnarray*}
		To simplify the notation, set
	\begin{equation*}
		B_k:=\bigg(\dfrac{e^{2}(d_{n+1}-d_{n})}{k-d_{n}}\bigg)^{k-d_{n}}.
	\end{equation*}
We next derive an upper bound for $B_k$. 
From the assumption (\ref{gap}), it follows that \begin{equation}\label{second}
	d_{n+1}-d_{n}\leq Ld_{n}-d_{n}=d_{n}(L-1),
\end{equation}
	and since $\sigma >1,$ we have
	\begin{equation}\label{ksigma}
		k<k^{\sigma}, \quad k\in \mathbb{N}.
	\end{equation} 
	By Lemma \ref{max1}, (\ref{second}) and (\ref{ksigma}), we obtain
 \begin{eqnarray}\label{first1}		
B_k=\bigg(\dfrac{e^{2}(d_{n+1}-d_{n})}{k-d_{n}}\bigg)^{k-d_{n}}&\leq& \sup_{x\geq 1} 	\bigg(\dfrac{e^{2}(d_{n+1}-d_{n})}{x}\bigg)^{x}\nonumber\\&\leq &\bigg(\dfrac{e^{2}(d_{n+1}-d_{n})}{e(d_{n+1}-d_{n})}\bigg)^{e(d_{n+1}-d_{n})} \\ &\leq& e^{e d_{n}(L-1)}\nonumber < (e^{e(L-1)})^{k^{\sigma}}\nonumber = h_1^{k^{\sigma}}, \nonumber
	\end{eqnarray} 
	where we set $h_1:=e^{e(L-1)}$.
    By the elementary estimate for the factorial, 
    $$(d_{n+1}-d_{n})!\leq (d_{n+1}-d_{n})^{d_{n+1}-d_{n}},$$ and the fact $0<1-\theta<1$, we obtain
	\begin{equation*}
\left((d_{n+1}-d_n)! \left( \frac{2}{b-a} \right)^{d_{n+1}-d_n}\right)^{1-\theta}
\le
\left( \frac{2(d_{n+1}-d_n)}{b-a} \right)^{d_{n+1}-d_n}.
\end{equation*}
	Next, we estimate \begin{equation*}
	\bigg(\frac{2(d_{n+1}-d_{n})}{b-a}\bigg)^{d_{n+1}-d_{n}}.
\end{equation*}
Since $d_{n+1}-d_{n}<d_{n+1}\leq Ld_{n}<Lk$ (see \eqref{gap}), we have \begin{equation*}
    	\bigg(\frac{2(d_{n+1}-d_{n})}{b-a}\bigg)^{d_{n+1}-d_{n}}\leq \bigg(\dfrac{2Lk}{b-a}\bigg)^{Lk}=(Sk)^{Lk},
\end{equation*}
where $S=\dfrac{2L}{b-a}$. By Lemma \ref{max2} and (\ref{ksigma}), we obtain 
\begin{eqnarray}\label{second2}
	(Sk)^{Lk}&=&e^{Lk\ln Sk}= e^{Lk\ln S +Lk\ln k}\leq e^{Lk^{\sigma}\ln S+LRk^{\sigma}} = h_2^{k^{\sigma}},
\end{eqnarray}
where $R$ is the constant from Lemma \ref{max2} and  $ h_2:=e^{L\ln S+LR}.$
	By (\ref{condition}), (\ref{first1}), (\ref{second2}) and the fact that $(M_{n}^{\tau ,\sigma})_{n\in \mathbb{N}_0}$ is an increasing sequence, we obtain \begin{eqnarray}\label{prva ocena}
		F_{k}\leq 2h_1^{k^{\sigma}}h_2^{k^{\sigma}}C_0 M_{k}^{\tau_0 , \sigma}=C'\tilde h^{k^{\sigma}}M_{k}^{\tau_0 , \sigma},
	\end{eqnarray} where $C'=2C_0$ and $\tilde h:=h_1h_2$, and the first case is proved.

%
%
{\bf Case 2.}

 \begin{equation*}
		\max \bigg\{(d_{n+1}-d_{n})!F_{d_{n}}\bigg( \frac{2}{b-a}\bigg)^{d_{n+1}-d_{n}},F_{d_{n+1}}\bigg\}=F_{d_{n+1}}.
\end{equation*} 
	By Lemma \ref{cglemma} and (\ref{condition}), we obtain
		\begin{eqnarray*}
			F_{k}&\leq&2B_kF_{d_{n}}^{\theta}F_{d_{n+1}}^{1-\theta} \\ &\leq& 2 B_k(C_0 d_{n}^{\tau_0 d_{n}^{\sigma}})^{\theta}(C_0 d_{n+1}^{\tau_0 d_{n+1}^{\sigma}})^{1-\theta}\\ &=&2B_kC_0d_{n}^{\tau_0 \theta d_{n}^{\sigma}}d_{n+1}^{\tau_0 (1-\theta)d_{n+1}^{\sigma}}.
		\end{eqnarray*}
Due to (\ref{gap}) and since $d_{n}<k<d_{n+1}$, we obtain $d_{n+1}\leq Ld_{n}<Lk$. Then it follows that \begin{equation*}
			\theta d_{n}^{\sigma}+(1-\theta)d_{n+1}^{\sigma}<\theta d_{n+1}^{\sigma}+(1-\theta)d_{n+1}^{\sigma}<(Lk)^{\sigma}.
		\end{equation*}
By taking the natural logarithm of $d_{n}^{\tau_0 \theta d_{n}^{\sigma}}d_{n+1}^{\tau_0 (1-\theta)d_{n+1}^{\sigma}}$, we obtain 
        \begin{eqnarray*}
			\ln d_{n}^{\tau_0 \theta d_{n}^{\sigma}}d_{n+1}^{\tau_0 (1-\theta)d_{n+1}^{\sigma}}&=&\tau_0 \theta d_{n}^{\sigma}\ln d_{n}+\tau_0 (1-\theta)d_{n+1}^{\sigma}\ln d_{n+1}\\ &\leq& \tau_0 (\theta d_{n}^{\sigma}+(1-\theta)d_{n+1}^{\sigma}) \ln d_{n+1}\leq \tau_0 (Lk)^{\sigma}\ln Lk.
		\end{eqnarray*}
		Then, taking the exponential of both sides of the above inequality, we get \begin{eqnarray}\label{dndn+1}
			d_{n}^{\tau_0 \theta d_{n}^{\sigma}}d_{n+1}^{\tau_0 (1-\theta)d_{n+1}^{\sigma}}&\leq& (Lk)^{\tau_0(Lk)^{\sigma}} = L^{\tau_0 L^{\sigma}k^{\sigma}}k^{\tau_0 L^{\sigma}k^{\sigma}}\nonumber \\ &=& h_3^{k^{\sigma}}k^{\tau' k^{\sigma}}  =h_3^{k^{\sigma}}M_{k}^{\tau',\sigma},
		\end{eqnarray}
		where $h_3:=L^{\tau_0 L^{\sigma}}$ and $\tau'=\tau_0 L^{\sigma}$. 
		From (\ref{first1}) and (\ref{dndn+1}), we conclude that 
		\begin{align}\label{druga ocena}
			F_{k}\leq 2h_1^{k^{\sigma}}C_0 h_3^{k^{\sigma}}M_{k}^{\tau',\sigma}
			 = C' \hat h^{k^\sigma}M_{k}^{\tau',\sigma},
		\end{align}
		where $C':=2C_0$ and $\hat h:=h_1 h_3$, which proves the second case.

        Finally, if $f$ satisfies \eqref{condition} form \eqref{prva ocena} and \eqref{druga ocena} we obtain
        \begin{equation}\label{C'h'}
            F_k\leq C' h'^{k^\sigma} M_k^{\tau',\sigma},\quad k\in\mathbb{N}_0,
        \end{equation}
        where $C':=2C_0$, $h':=\max\{\tilde h,\hat h\}$ and $\tau'=\tau_0 L^{\sigma}$.
        This, with Theorem \ref{thm:with-or-without-h} or more precisely inequality \eqref{jede h}, implies
         \begin{equation*}\label{eq:final_est}
            F_k\leq C M_{k}^{\tau,\sigma},\quad k\in\mathbb{N}_0,
        \end{equation*}
        where $C=C'A$ and $\tau=2\tau'$ and constant $A$ is form inequality \eqref{jede h}. This completes our proof.
	\end{proof}
    
The following result is the Beurling counterpart of Theorem \ref{thm: main}. Its proof follows along the same lines as that of Theorem \ref{thm: main}, with the necessary modifications, and is therefore omitted.

	\begin{theorem}
		Let $\sigma >1$ be given. Let $f\in C^{\infty}([a,b])$ and let $(d_{n})_{n\in \mathbb{N}}$  be an increasing divergent sequence of positive integers satisfying \begin{equation*}\label{gap*}
			(\exists L>1) \quad \dfrac{d_{n+1}}{d_{n}}\leq L, \quad n\in \mathbb{N}_0.
		\end{equation*}
		Assume that \begin{equation*}\label{condition*}
			(\forall \tau_0 >0)(\exists C_0>0) \quad \max_{x\in[a,b]}|f^{(d_n)}(x)|\leq C_0 M_{d_{n}}^{\tau_0 , \sigma}, \quad n\in \mathbb{N}_0.
		\end{equation*}
		Then $f\in \mathcal E_{(\sigma)}([a,b])$.
	\end{theorem}

\begin{remark}\label{rem 3.6}
    Let us note that the existence of the interpolation in the setting of extended Gevrey spaces could have been shown by applying \cite[Theorem 5.3]{RS2024}, in view of the fact that the supergeometric factor $h^{n^\sigma}$ can be disregarded, as established in Theorem \ref{thm:with-or-without-h}. Nevertheless, we have employed a different approach here in order to provide an explicit description of the passage from the parameters $\tau_0$ and $C_0$ to the resulting parameters $\tau$ and $C$. 
\end{remark}

\section{Applications to extended Gelfand-Shilov spaces}
\label{sec:4}

Gelfand-Shilov type spaces are among the most important and well-known subspaces of Gevrey type spaces that are invariant under the Fourier transform. In what follows, we introduce extended Gelfand–Shilov spaces and prove analogous interpolation theorems in this setting.

\subsection{Extended Gelfand-Shilov spaces}

Let $\sigma>1$ and $\tau>0$. Let $(M_{n}^{\tau,\sigma})_{n\in\mathbb N_0}$ be the weight sequences given by \eqref{eq:the-sequence}. In order to be consistent with the existing literature and the results already developed in this direction, for fixed $\sigma>1$ we will use the notation for the corresponding weight matrix $\mathcal{M}_\sigma$, given by
    \begin{equation}\label{eq:the-weight-matrix}
\mathcal{M}_\sigma=\left\{(M_{n}^{\tau,\sigma})_{n\in\mathbb N_0}\colon \tau>0\right\}.
    \end{equation}

\begin{definition}\label{Def:extGSs}
Let $\sigma>1$ and the weight matrix $\mathcal{M}_\sigma$ be given by 
\eqref{eq:the-weight-matrix}. Then
the extended Gelfand-Shilov space of Roumieu type related to $\mathcal{M}_\sigma$ is given by
\begin{align*}
 \mathcal S_{\{\mathcal{M}_\sigma\}}^{\{\mathcal{M}_\sigma\}} (\mathbb R)= \{\varphi\in C^{\infty}(\mathbb R)\colon  (\exists \tau>0)\;(\exists C>0)&\;  (\forall \alpha,\beta\in\mathbb N_0) \\
  &\sup_{x\in\mathbb R}|x^\alpha \varphi^{(\beta)} (x)|\leq C M^{\tau,\sigma}_{\alpha} 
   M^{\tau,\sigma}_{\beta}\},
\end{align*}
and the extended Gelfand-Shilov space of Beurling type related to $\mathcal{M}_\sigma$ is given by
\begin{align*}
 \mathcal S_{(\mathcal{M}_\sigma)}^{(\mathcal{M}_\sigma)} (\mathbb R)= \{\varphi\in C^{\infty}(\mathbb R)\colon(\forall \tau>0)\;(\exists C>0)&\;  (\forall \alpha,\beta\in\mathbb N_0) \\
  &\sup_{x\in\mathbb R}|x^\alpha  \varphi^{(\beta)} (x)|\leq C M^{\tau,\sigma}_{\alpha} 
   M^{\tau,\sigma}_{\beta}\}.
\end{align*}
\end{definition}
We use $\mathcal S_{\mathcal{M}_\sigma}^{\mathcal{M}_\sigma} (\mathbb R)$ as a common notation for
$\mathcal S_{\{\mathcal{M}_\sigma\}}^{\{\mathcal{M}_\sigma\}} (\mathbb R)$
or $ \mathcal S_{(\mathcal{M}_\sigma)}^{(\mathcal{M}_\sigma)} (\mathbb R)$.

Topologically, these spaces are inductive (respectively, projective) limits of Banach spaces. More precisely, the extended Gelfand-Shilov space of Roumieu type is a $(DF)$-space, whereas its Beurling counterpart is an $(F)$-space. Moreover, in \cite[Theorem~3]{TTZ2025}, the authors additionally establish the nuclearity of these spaces. 

The associated function of the sequence $M^{\tau,\sigma}_n$, $n\in \mathbb{N}_0$, is defined by
\begin{equation*}
    \label{eq:associated-function}
T_{\tau,\sigma}(x)=\sup_{n\in \mathbb{N}_0}\ln \frac{x^n}{M^{\tau, \sigma}_n}, \qquad
x\geq 0. 
\end{equation*}
We assume, by convention that $0^0=1$ so that $T_{\tau,\sigma}(0)=0.$
Observe that, for every $x\in(0,1]$ and every $n\in \mathbb{N}_0$ we have $\ln \frac{x^n}{M^{\tau, \sigma}_n}\leq 0$ with equality when $n=0$. Therefore, $T_{\tau,\sigma}(x)=0$ for $x\in [0,1]$. 

For the reader's convenience, we recall an important theorem providing equivalent characterizations of Gelfand--Shilov type spaces in terms of the growth and decay properties of a function and its Fourier transform, as well as the decay described by the associated function, in a form analogous to the corresponding result in \cite{Chung 1996}. Recall that the Fourier transform of $\varphi \in \mathcal S_{\mathcal{M}_\sigma}^{\mathcal{M}_\sigma} (\mathbb R)$  is given by
\begin{align*}
    (\mathcal{F}\varphi)(\xi)=\widehat \varphi(\xi)=\int_{\mathbb R}\varphi(x)e^{-2\pi ix\xi}\;dx,\quad \xi\in \mathbb R.
\end{align*}

\begin{theorem}\label{conditions}\cite[Theorem~5]{TTZ2025}
Let $\sigma>1$ and the weight matrix $\mathcal{M}_\sigma$ be given by 
\eqref{eq:the-weight-matrix}.
If $ \varphi\in  C^\infty (\mathbb R) $, then
the following conditions are equivalent:
    \begin{enumerate}
        \item $\varphi\in \mathcal  S_{\{\mathcal{M}_\sigma\}}^{\{\mathcal{M}_\sigma\}} (\mathbb R);$
        \item $(\exists \tau>0)(\exists C>0)\;(\forall \alpha,\beta\in \mathbb N_0)$
        $$\sup_{x\in\mathbb R}|x^\alpha \varphi(x)|\leq C M^{\tau,\sigma}_{\alpha} \quad \text{and}\quad \sup_{x\in\mathbb R}|\varphi^{(\beta)}(x)|\leq CM^{\tau,\sigma}_{\beta};$$
        \item $(\exists \tau>0)(\exists C>0)\;(\forall \alpha,\beta\in \mathbb N_0)$
        $$\sup_{x\in\mathbb R}|x^\alpha \varphi(x)|\leq C M^{\tau,\sigma}_{\alpha} \quad \text{and}\quad \sup_{\xi\in\mathbb R}|\xi^\beta \widehat\varphi(\xi)|\leq CM^{\tau,\sigma}_{\beta};$$
        \item $(\exists \tau>0)(\exists C>0)\;(\forall \alpha,\beta\in \mathbb N_0)$
        $$\sup_{\xi\in\mathbb R}|\widehat\varphi^{(\alpha)}(\xi)|\leq C M^{\tau,\sigma}_{\alpha} \quad \text{and}\quad \sup_{\xi\in\mathbb R}|\xi^\beta \widehat\varphi(\xi)|\leq CM^{\tau,\sigma}_{\beta};$$
        \item $(\exists \tau>0)$
        $$\sup_{x\in\mathbb R} |\varphi(x)|e^{T_{\tau,\sigma}(|x|)}<\infty \quad \text{and}\quad \sup_{\xi\in\mathbb R} |\widehat{\varphi}(\xi)|e^{T_{\tau,\sigma}(|\xi|)}<\infty.$$
    \end{enumerate}
\end{theorem}

The Beurling counterpart of this theorem also holds (see \cite[Theorem~6]{TTZ2025}), where each occurrence of $(\exists \tau>0)$ in the previous theorem is replaced by $(\forall \tau>0)$.

\subsection{Interpolation problem in extended Gelfand-Shilov spaces}

In this section, we extend the results established in Section \ref{sec:3} for extended Gevrey spaces to the setting of extended Gelfand-Shilov spaces. More precisely, we show that estimates involving powers and derivatives at the orders of a suitably chosen increasing sequence imply the corresponding estimates for all orders.

Additionally, we establish a result concerning the associated function, which is of independent interest. More precisely, we show that it is sufficient to take the supremum over a suitably chosen subsequence in the definition of the associated function in order to obtain estimates equivalent to those given by its standard definition.

Let 
 $(\alpha_{n})_{n\in \mathbb{N}_{0}}$ and $(\beta_{m})_{m\in \mathbb{N}_{0}}$ be increasing  sequences of non-negative integers such that 
 \begin{gather}
\lim_{n\to\infty} \alpha_n=\lim_{n\to\infty} \beta_n=\infty,\quad
\alpha_0=\beta_0=0, \notag\\
(\exists L>1)\quad  \dfrac{\alpha_{n+1}}{\alpha_n}\le L,\qquad
\dfrac{\beta_{n+1}}{\beta_n}\le L,\quad n\in\mathbb{N}.
\label{firstcondition}
\end{gather}

\begin{theorem}\label{12ekv}
		Let $\sigma >1$ be given,  $(\alpha_{n})_{n\in \mathbb{N}_{0}}$ and $(\beta_{m})_{m\in \mathbb{N}_{0}}$ be increasing  sequences of non-negative integers for which \eqref{firstcondition} holds. 
    Suppose that  $\varphi \in C^{\infty}(\mathbb R)$ satisfy the following estimate: 
 \begin{equation}\label{conditionsup}
	(\exists \tau_0 >0)(\exists C_0>0)\quad \sup_{x\in \mathbb R}|x^{\alpha_{n}}\varphi^{(\beta_{m})} (x)|\leq C_0M_{\alpha_{n}}^{\tau_0 , \sigma}M_{\beta_{m}}^{\tau_0 , \sigma},\quad m,n\in \mathbb{N}_{0}.
\end{equation}
Then, $\varphi \in \mathcal S_{\{M_{\sigma}\}}^{\{M_{\sigma}\}}(\mathbb R).$
	\end{theorem}

\begin{proof}
		Let $\varphi \in C^{\infty}(\mathbb R)$. We show that the conditions in part 2. of Theorem \ref{conditions} hold. 

First we prove the estimate for the derivatives of $\varphi$. By letting $n=0$ in (\ref{conditionsup})  we obtain \begin{equation}\label{eq:k}
	\sup_{x\in [k,k+1]}|\varphi ^{(\beta_{m})}(x)|\leq C_0M_{\beta_{m}}^{\tau_0 ,\sigma}, \quad m\in \mathbb{N}_{0},
\end{equation}
for every $k\in\mathbb Z$.

Fix $k\in\mathbb Z$. Since the sequence $(\beta_{m})$ satisfies (\ref{firstcondition}), 
Theorem \ref{thm: main} applied to the restriction of the  function $\varphi$  to the interval $[k,k+1]$ yields the existence of constants $C',\tau ' >0$ independent of $k$ such that \begin{equation}\label{kk+1}
	\sup_{x\in [k,k+1]}|\varphi^{(\beta)}(x)|\leq C'M_{\beta}^{\tau', \sigma}, \quad \beta \in \mathbb{N}_{0}.
\end{equation}
A careful inspection of the proof of Theorem \ref{thm: main} shows that the constants $C$ and $\tau$ appearing in \eqref{C'h'} depend solely on the length $b-a$ of the interval, rather than on its endpoints. 
Since all intervals $[k,k+1]$ appearing in \eqref{eq:k} have the same length, the constants $C'$ and $\tau'$ in \eqref{kk+1} may be chosen independently of $k$, for all $k\in\mathbb Z$. Thus we yield 
\begin{equation}\label{eq 1}
\sup_{x\in \mathbb R}|\varphi^{(\beta)}(x)|\leq C'M_{\beta}^{\tau', \sigma}, \quad \beta \in \mathbb{N}_{0}.
\end{equation}

It remains to prove that for every   $\alpha\in\mathbb N_0$ we have $\displaystyle \sup_{x\in\mathbb R} |x^{\alpha}\varphi(x)|\leq C''  M^{\tau'',\sigma}_{\alpha}$ for some $C'',\tau''>0.$ 

Fix $\alpha\in \mathbb N$. Let $n\in\mathbb N$ satisfies $\alpha_n\leq \alpha< \alpha_{n+1}$. If 
$\alpha=\alpha_n$ then the estimate \eqref{conditionsup} is enough to conclude the proof. Suppose now $\alpha\in(\alpha_n,\alpha_{n+1})$. As in the proof of Theorem \ref{thm: main}, set $\theta=\frac{\alpha_{n+1}-\alpha}{\alpha_{n+1}-\alpha_n}$. Then it is easy to see that 
$\theta\in (0,1)$ and
$\alpha=\theta \alpha_n+(1-\theta)\alpha_{n+1}$. Fix $x\in \mathbb R$. 

By raising the inequality $|x^{\alpha_n}\varphi(x)|\leq  C_0 M^{\tau_0,\sigma}_{\alpha_n}$ to the power $\theta$ and  inequality $|x^{\alpha_{n+1}}\varphi(x)|\leq  C_0 M^{\tau_0,\sigma}_{\alpha_{n+1}}$ to the power $1-\theta$ and then multiplying them, we obtain 
$$|x^{\theta \alpha_n+(1-\theta)\alpha_{n+1}}\varphi(x)|\leq C_0 \alpha_n^{\theta\tau_0 \alpha_n^{\sigma} }\cdot  \alpha_{n+1}^{(1-\theta)\tau_0 \alpha_{n+1}^{\sigma} } .$$
The estimate \eqref{dndn+1} applied to the sequence $(\alpha_n)_{n\in\mathbb N_0}$ yields 
 $$ \alpha_n^{\theta\tau_0 \alpha_n^{\sigma} }\cdot  \alpha_{n+1}^{(1-\theta)\tau_0 \alpha_{n+1}^{\sigma} }\leq h^{\alpha^{\sigma}} \alpha^{\tau_0 L^{\sigma} \alpha^{\sigma}}=h^{\alpha^{\sigma}} M^{\tau_0 L^{\sigma},\sigma}_\alpha,$$ where $h=L^{\tau_0 L^{\sigma}}.$  Finally, from \eqref{jede h} we have 
 $$|x^{\alpha}\varphi(x)|\leq C_0 \alpha_n^{\theta\tau_0 \alpha_n^{\sigma} }\cdot  \alpha_{n+1}^{(1-\theta)\tau_0 \alpha_{n+1}^{\sigma} }\leq C_0 h^{\alpha^{\sigma}} M^{\tau_0 L^{\sigma},\sigma}_\alpha\leq C'' M^{\tau'',\sigma}_\alpha,$$ where 
 $\tau''=2\tau_0 L^{\sigma}$ and $C''=C_0A>0$, with $A$ denoting the constant from \eqref{jede h}. Since $x\in\mathbb R$ was chosen arbitrary, we have also 
 \begin{equation}
     \label{eq 2}
     \sup_{x\in\mathbb R} |x^\alpha \varphi(x)|\leq  C'' M^{\tau'',\sigma}_\alpha, \quad \alpha \in \mathbb{N}_{0}.
 \end{equation}
 
Now, taking $\tau=\max\{\tau',\tau''\}$ and $C=\max\{C',C''\}$, the proof follows from \eqref{eq 1}, \eqref{eq 2}, and Theorem \ref{conditions} part $1.\Leftrightarrow 2.$
	\end{proof}

Since we have explained how the constants $\tau'$ and $\tau''$ arise, it is straightforward to verify that the corresponding Beurling counterpart also holds.

\begin{theorem}
		Let $\sigma >1$ be given,  $(\alpha_{n})_{n\in \mathbb{N}_{0}}$ and $(\beta_{m})_{m\in \mathbb{N}_{0}}$ be increasing  sequences of non-negative integers for which \eqref{firstcondition} holds. 
    Suppose that  $\varphi \in C^{\infty}(\mathbb R)$ satisfy the following estimate:  
 \begin{equation*}\label{conditionsupbeurling}
	(\forall \tau_0 >0)(\exists C_0>0)\quad \sup_{x\in \mathbb R}|x^{\alpha_{n}}\varphi^{(\beta_{m})} (x)|\leq C_0 M_{\alpha_{n}}^{\tau_0 , \sigma}M_{\beta_{m}}^{\tau_0 , \sigma},\quad m,n\in \mathbb{N}_{0}.
\end{equation*}
Then, $\varphi \in \mathcal S_{(M_{\sigma})}^{(M_{\sigma})}(\mathbb R).$
	\end{theorem}

In what follows, we will show that a corresponding statement can also be established for associated functions. We begin with the following technical lemma.

For $\sigma>1$, $\tau>0$, and an arbitrary sequence $(d_n)_{n\in\mathbb N_0}$ satisfying
$\lim_{n\to\infty}d_n=\infty$ and $d_0=0$, define
$$T_{\tau,\sigma,(d_n)}(x)=\sup_{n\in\mathbb N_0} \ln \frac{x^{d_n}}{M^{\tau,\sigma}_{d_n}},\quad x\geq 0.$$
As in the case of the associated function, we have $T_{\tau,\sigma,(d_n)}(x)=0$ for $x\in [0,1].$ 

\begin{lemma}\label{15ekv}
Let $\sigma>1$, $\tau>0$, and let $(d_n)_{n\in\mathbb{N}_0}$ be an increasing divergent sequence of positive integers satisfying \eqref{gap}. Then, for every $\tau>0$, there exist $\tau_1,C_1>0$ such that
\begin{equation}\label{associated_gap}
T_{\tau,\sigma,(d_n)}(x)\leq T_{\tau,\sigma}(x)\leq  T_{\tau_1,\sigma,(d_n)}(x)+C_1,\quad x\geq 0,
\end{equation}
where the constant $C_1$ depends on $\sigma$, $\tau$ and $L$.
\end{lemma}
\begin{proof}
For $x\in[0,1]$, the inequality \eqref{associated_gap} trivially holds. Hence, it remains to consider the case $x>1$. The first inequality is immediate. We now turn to the proof of the second inequality.

Let $p\in\mathbb N_0$ be fixed, and choose $n\in\mathbb N_0$ such that $p\in[d_n,d_{n+1})$. When $p=d_n$, the estimate \eqref{pomocna} trivially holds.
Next, suppose that $p\in(d_n,d_{n+1})$. As in the proof of Theorem \ref{thm: main}, set $\theta=\frac{d_{n+1}-p}{d_{n+1}-d_n}$. Then $\theta\in(0,1)$ and
$p=\theta d_n+(1-\theta)d_{n+1}$. 

The estimate \eqref{dndn+1} applied to the sequence $(d_n)_{n\in\mathbb N_0}$ yields 
 $$ d_n^{\theta\tau d_n^{\sigma} }\cdot  d_{n+1}^{(1-\theta)\tau d_{n+1}^{\sigma} }\leq h^{p^{\sigma}} p^{\tau L^{\sigma} p^{\sigma}}=h^{p^{\sigma}} M^{\tau L^{\sigma},\sigma}_p,$$ 
 where $h=L^{\tau L^{\sigma}}>1.$  Since the function $t\mapsto t^{\sigma}$ is convex on $[0,\infty)$, we infer that
 $p^{\sigma}\leq \theta d^{\sigma}_n+(1-\theta)d^{\sigma}_{n+1} $. Thus, we obtain 
 $$\frac{x^p}{M^{\tau,\sigma}_p}\leq \frac{x^p h^{p^{\sigma}}}{d_n^{\theta\tau_2 d_n^{\sigma} }\cdot  d_{n+1}^{(1-\theta)\tau_2 d_{n+1}^{\sigma} }}\leq \left(\frac{x^{d_n} h^{d^{\sigma}_n}  }{d_n^{\tau_2 d^{\sigma}_n}}\right)^{\theta}\cdot\left(\frac{x^{d_{n+1}} h^{d^{\sigma}_{n+1}}  }{d_{n+1}^{\tau_2 d^{\sigma}_{n+1}}}\right)^{1-\theta},$$
 where $\tau_2=\frac{\tau}{L^{\sigma}}.$
 From \eqref{jede h}, we have $\frac{h^{n^{\sigma}}}{M^{\tau_2,\sigma}_n}\leq  \frac{A}{M^{\tau_2/2,\sigma}_n}$,  for some $A>0$ depending only on $\sigma$, $\tau$ and $L$, and for every $n\in\mathbb N_0$. Suppose, without loss of generality, that $A\geq 1$. Then 
$$\frac{x^p}{M^{\tau,\sigma}_p}\leq A \left( \frac{x^{d_n} }{M^{\tau_2/2,\sigma}_{d_n}}   \right)^{\theta}\cdot  \left(\frac{x^{d_{n+1}}}{M^{\tau_2/2,\sigma}_{d_{n+1}}}\right)^{1-\theta}$$ 
and 
\begin{equation}\label{pomocna}
\ln \frac{x^p}{M^{\tau,\sigma}_p}\leq \ln A+\theta T_{\tau_2/2,\sigma,(d_n)}(x)+(1-\theta) T_{\tau_2/2,\sigma,(d_n)}(x) \leq T_{\tau_2/2,\sigma,(d_n)}(x)+C_1 ,\end{equation}
where $C_1=\ln A.$

 Hence, \eqref{pomocna} holds for an arbitrary $p\in\mathbb N_0$. Therefore, the second inequality in \eqref{associated_gap} is valid for $\tau_1=\tau_2/2=\frac{\tau}{2L^\sigma}$, with a constant $C_1$ depending only on $\sigma$, $\tau$ and $L$.  
\end{proof}

Finally, we establish an interpolation theorem for extended Gelfand–Shilov spaces, valid for all equivalent characterizations of these spaces.

\begin{theorem}\label{conditions gap}
Let $\sigma>1$ and the weight matrix $\mathcal{M}_\sigma$ be given by 
\eqref{eq:the-weight-matrix}.
If $ \varphi\in  C^\infty (\mathbb R) $, then
the following conditions are equivalent:
    \begin{enumerate}
        \item $\varphi\in \mathcal  S_{\{\mathcal{M}_\sigma\}}^{\{\mathcal{M}_\sigma\}} (\mathbb R).$
        \item  There exists increasing sequences $(\alpha_{n})_{n\in \mathbb{N}_{0}}$ and $(\beta_{m})_{m\in \mathbb{N}_{0}}$  of non-negative integers  satisfying \eqref{firstcondition} and $\tau,C>0$ such that  for all $m,n\in\mathbb N_0$:
        $$\sup_{x\in\mathbb R}|x^{\alpha_n} \varphi(x)|\leq C M^{\tau,\sigma}_{\alpha_n} \quad \text{and}\quad \sup_{x\in\mathbb R}|\varphi^{(\beta_m)} (x)|\leq CM^{\tau,\sigma}_{\beta_m}.$$
        \item There exists increasing sequences $(\alpha_{n})_{n\in \mathbb{N}_{0}}$ and $(\beta_{m})_{m\in \mathbb{N}_{0}}$  of non-negative integers  satisfying \eqref{firstcondition} and $\tau,C>0$ such that  for all $m,n\in\mathbb N_0$:
        $$\sup_{x\in\mathbb R}|x^{\alpha_n} \varphi(x)|\leq C M^{\tau,\sigma}_{\alpha_n}\quad \text{and}\quad \sup_{\xi\in\mathbb R}|\xi^{\beta_m} \widehat\varphi(\xi)|\leq CM^{\tau,\sigma}_{\beta_m}.$$
        \item There exists increasing sequences $(\alpha_{n})_{n\in \mathbb{N}_{0}}$ and $(\beta_{m})_{m\in \mathbb{N}_{0}}$  of non-negative integers satisfying \eqref{firstcondition} and $\tau,C>0$ such that  for all $m,n\in\mathbb N_0$:
        $$\sup_{\xi \in\mathbb R}|\widehat\varphi^{(\alpha_n)} (\xi)|\leq C M^{\tau,\sigma}_{\alpha_n} \quad \text{and}\quad \sup_{\xi\in\mathbb R}|\xi^{\beta_m} \widehat\varphi(\xi)|\leq CM^{\tau,\sigma}_{\beta_m}.$$

        \item There exists increasing sequences $(\alpha_{n})_{n\in \mathbb{N}_{0}}$ and $(\beta_{m})_{m\in \mathbb{N}_{0}}$  of non-negative integers  satisfying \eqref{firstcondition} 
        such that for some $\tau>0$ we have $$\sup_{x\in \mathbb R}  |\varphi(x)| e^{T_{\tau,\sigma,(\alpha_n)}(|x|)}<\infty \quad \text{and}\quad  \sup_{\xi\in \mathbb R}  |\widehat\varphi(\xi)| e^{T_{\tau,\sigma,(\beta_n)}(|\xi|)}<\infty.   $$
    \end{enumerate}
\end{theorem}
\begin{proof}
The implication $1.\Rightarrow 2.$ follows from the corresponding implication in Theorem \ref{conditions}, whereas the converse implication follows from Theorem \ref{12ekv}. 

The fact that  ${\mathcal F}{\mathcal S}_{\{\mathcal{M}_\sigma\}}^{\{\mathcal{M}_\sigma\}} (\mathbb R)={\mathcal S}_{\{\mathcal{M}_\sigma\}}^{\{\mathcal{M}_\sigma\}} (\mathbb R)$  (\cite[Theorem~4]{TTZ2025}) together with the  equivalence $1.\Leftrightarrow 2.$
prove  $1.\Leftrightarrow 4.$  

 Moreover, the equivalence $1.\Leftrightarrow 5.$ follows from Lemma~\ref{15ekv} and the corresponding equivalence in Theorem~\ref{conditions}.

Finally, $1.\Rightarrow 3.$ follows from Theorem \ref{conditions}, whereas $3.\Rightarrow 1.$ follows from the same arguments as in Theorem \ref{12ekv}, therefore the proof is omitted.
\end{proof}

\begin{theorem}\label{conditions gap Beurling}
Let $\sigma>1$ and the weight matrix $\mathcal{M}_\sigma$ be given by 
\eqref{eq:the-weight-matrix}.
If $ \varphi\in  C^\infty (\mathbb R) $, then
the following conditions are equivalent:
    \begin{enumerate}
        \item $\varphi\in \mathcal  S_{(\mathcal{M}_\sigma)}^{(\mathcal{M}_\sigma)} (\mathbb R).$
        \item  There exists increasing sequences $(\alpha_{n})_{n\in \mathbb{N}_{0}}$ and $(\beta_{m})_{m\in \mathbb{N}_{0}}$  of non-negative integers  satisfying \eqref{firstcondition} and for every $\tau>0$ there exists $C>0$ such that  for all $m,n\in\mathbb N_0$:
        $$\sup_{x\in\mathbb R}|x^{\alpha_n} \varphi(x)|\leq C M^{\tau,\sigma}_{\alpha_n} \quad \text{and}\quad \sup_{x\in\mathbb R}|\varphi^{(\beta_m)}(x)|\leq CM^{\tau,\sigma}_{\beta_m}.$$
        \item There exists increasing sequences $(\alpha_{n})_{n\in \mathbb{N}_{0}}$ and $(\beta_{m})_{m\in \mathbb{N}_{0}}$  of non-negative integers  satisfying \eqref{firstcondition} and for every $\tau>0$ there exists $C>0$ such that  for all $m,n\in\mathbb N_0$:
        $$\sup_{x\in\mathbb R}|x^{\alpha_n} \varphi(x)|\leq C M^{\tau,\sigma}_{\alpha_n}\quad \text{and}\quad \sup_{\xi\in\mathbb R}|\xi^{\beta_m} \widehat\varphi(\xi)|\leq CM^{\tau,\sigma}_{\beta_m}.$$
        \item There exists increasing sequences $(\alpha_{n})_{n\in \mathbb{N}_{0}}$ and $(\beta_{m})_{m\in \mathbb{N}_{0}}$  of non-negative integers satisfying \eqref{firstcondition} and for every $\tau>0$ there exists $C>0$ such that  for all $m,n\in\mathbb N_0$:
        $$\sup_{\xi \in\mathbb R}|\widehat\varphi^{(\alpha_n)} (\xi)|\leq C M^{\tau,\sigma}_{\alpha_n} \quad \text{and}\quad \sup_{\xi\in\mathbb R}|\xi^{\beta_m} \widehat\varphi(\xi)|\leq CM^{\tau,\sigma}_{\beta_m}.$$

        \item There exists increasing sequences $(\alpha_{n})_{n\in \mathbb{N}_{0}}$ and $(\beta_{m})_{m\in \mathbb{N}_{0}}$  of non-negative integers  satisfying \eqref{firstcondition} 
        such that for every  $\tau>0$ we have 
        $$\sup_{x\in \mathbb R}  |\varphi(x)| e^{T_{\tau,\sigma,(\alpha_n)}(|x|)}<\infty \quad \text{and}\quad  \sup_{\xi\in \mathbb R}  |\widehat\varphi(\xi)| e^{T_{\tau,\sigma,(\beta_n)}(|\xi|)}<\infty .  $$      
    \end{enumerate}
\end{theorem}

\section*{Appendix}

\begin{proof}[\textbf{Proof of Lemma \ref{max1}}]
	We define $\psi (x):=\ln \phi (x),$ $x\geq 1$. Then, 
    \begin{equation*}
		\psi(x)=\ln \bigg(\dfrac{D}{x}\bigg)^{x}
        =x(\ln D-\ln x),\quad x\geq 1.
	\end{equation*}
	Since the natural logarithm is strictly increasing, functions $\phi$  and $\psi$ attain their extrema at the same points.
    We obtain \begin{equation*}
		\psi '(x)=\ln D-\ln x-1,\quad x\geq 1.
	\end{equation*}
		Hence, $\psi '(x)=0$ if and only if $\ln x=\ln D-1,$ which yields the unique critical point $x=\dfrac{D}{e}.$  Since $D>e,$ we have $\dfrac{D}{e}<1.$
		Moreover, the second derivative satisfies 
        \begin{equation*}
			\psi ''(x)=-\dfrac{1}{x}<0, \quad x\geq 1.
		\end{equation*}
		Therefore, $\psi$ is strictly concave on 
        $[1,\infty)$ and
        its unique critical point coincides with  the point where  $\psi$ attains its maximum, namely $x=\dfrac{D}{e}$.
	\end{proof}

	\begin{proof}[\textbf{Proof of Lemma \ref{max2}}]
		Consider the function \begin{equation*}
			\varphi (x):=\dfrac{x\ln x}{x^{\sigma}},\quad x\geq 1.
		\end{equation*}
		The first derivative of the function $\varphi$ is given by \begin{equation*}
			\varphi '(x)=(x^{1-\sigma}\ln x)'=(1-(\sigma -1)\ln x)x^{-\sigma},\quad x\geq 1.
		\end{equation*}
		Therefore, $\varphi '(x)=0$ if and only if $1-(\sigma -1)\ln x=0$ which yields the unique critical point $x=e^{\frac{1}{\sigma -1}}.$ Differentiating once more, we obtain \begin{eqnarray*}
			\varphi ''(x)=((1-(\sigma -1)\ln x)x^{-\sigma})'=\dfrac{\sigma(\sigma -1)\ln x-(2\sigma -1)}{x^{\sigma +1}},\quad x\geq 1
        .
		\end{eqnarray*}
		Since $\sigma >1$, the second derivative at the critical point is 
        \begin{eqnarray*}
			\varphi ''(e^{\frac{1}{\sigma -1}})=\dfrac{\sigma (\sigma -1)\ln e^{\frac{1}{\sigma -1}}-(2\sigma -1)}{(e^{\frac{1}{\sigma -1}})^{\sigma +1}}=\dfrac{1-\sigma}{e^{\frac{\sigma +1}{\sigma -1}}}<0.
		\end{eqnarray*}
		It follows that $\varphi$ attains its maximum at $x=e^{\frac{1}{\sigma -1}}$. Thus, \begin{eqnarray*}
			\dfrac{x\ln x}{x^{\sigma}} =\dfrac{\ln x}{x^{\sigma -1}} \leq \dfrac{1}{e(\sigma -1)},\quad x\geq 1,
		\end{eqnarray*}
		that is, $x\ln x \leq \dfrac{1}{e(\sigma -1)}x^{\sigma} =Rx^{\sigma},$ $x\geq 1$.
	\end{proof}

\subsection*{Declaration of competing interest}
The authors have no competing interests to declare.


\subsection*{Acknowledgment}
The authors thank N. Teofanov for helpful discussions.

\noindent
The second and third authors were supported by the Science Fund of the Republic of Serbia, Grant No. 2727, {\it Global and Local Analysis of Operators and Distributions} – GOALS, and the third author by the Ministry of Science, Technological Development and Innovation of the Republic of Serbia,
Grants No. 451-03-33/2026-03/200125 and 451-03-34/2026-03/200125.


\end{document}